\documentclass[11pt,a4paper]{article}
\usepackage[T1]{fontenc}
\usepackage[utf8]{inputenc}
\usepackage{lmodern}
\usepackage[margin=27mm]{geometry}
\usepackage{amsmath,amssymb,amsthm,mathtools}
\usepackage{microtype,enumitem,needspace}
\usepackage{xcolor}
\usepackage{hyperref}
\hypersetup{colorlinks=true,linkcolor=blue!45!black,citecolor=blue!45!black,
urlcolor=blue!45!black,
pdftitle={Functions determining (q+3)/2 directions over the finite field with q elements},
pdfauthor={Bence Csajb\'ok}}
\numberwithin{equation}{section}
\newtheorem{theorem}{Theorem}[section]
\newtheorem{lemma}[theorem]{Lemma}
\newtheorem{proposition}[theorem]{Proposition}
\newtheorem{corollary}[theorem]{Corollary}
\theoremstyle{remark}

\newtheorem*{remark*}{Remark}
\newcommand{\F}{\mathbb F_q}
\newcommand{\Fp}{\mathbb F_p}
\newcommand{\AG}{\operatorname{AG}}
\newcommand{\ord}{\operatorname{ord}}
\newcommand{\Stab}{\operatorname{Stab}}
\newcommand{\UU}{\mathcal U}
\setlist[enumerate]{itemsep=3pt,topsep=5pt}
\allowdisplaybreaks[1]
\title{Functions determining $(q+3)/2$ directions\\over the finite field with $q$ elements}
\author{Bence Csajb\'ok\thanks{Department of Computer Science,
ELTE E\"otv\"os Lor\'and University, H-1117 Budapest,
P\'azm\'any P. stny. 1/C, Hungary.
Email: \href{mailto:bence.csajbok@ttk.elte.hu}{\nolinkurl{bence.csajbok@ttk.elte.hu}}.}}
\date{22 October 2026}
\begin{document}
\maketitle
\vspace{-1.5em}
\begin{abstract}
Let $q=p^n$, where $p>2$ is prime. We prove that every $q$-set
$U\subseteq\AG(2,q)$ determining exactly $(q+3)/2$ directions is affinely
equivalent to the graph of $x^{(q+1)/2}$. For $n\ge2$, we also prove
that a $q$-set determining at most $(q+p)/2$ directions is
either an affine $\Fp$-linear set or affinely equivalent to this graph.
The classification for $n=1$ is due to Lov\'asz and Schrijver.
For $n=2$, the classification and the stated stability result follow
from results of G\'acs, Lov\'asz, and Sz\H onyi and the direction
theorems of Blokhuis, Ball, Brouwer, Storme, and Sz\H onyi, completed
by Ball. These direction theorems establish affine $\Fp$-linearity
for functions with fewer than $(q+3)/2$ determined directions.
Our proof uses R\'edei polynomials, results of G\'acs, Lov\'asz, and
Sz\H onyi on their coefficients, and an earlier polynomial comparison
lemma of the author.
\end{abstract}

\section{Introduction}
Let $q=p^n$, where $p>2$ is prime, and let $\F$ be the field of $q$
elements. The affine plane $\AG(2,q)$ extends to the projective plane
$\operatorname{PG}(2,q)$ by adding $q+1$ points on a new line
$\ell_\infty$, called the line at infinity. The points of
$\ell_\infty$ are called \emph{directions} and correspond to the
parallel classes of affine lines. An affine line $\ell$ has direction
$d\in\ell_\infty$, written $\mathrm{dir}(\ell)=d$, if its projective
extension contains $d$. In
$\AG(2,q)$, a parallel class consists of the lines with the same
slope, so we also identify directions with the elements of
$\F\cup\{\infty\}$, where $\infty$ denotes the vertical direction.

A point set $U\subseteq\AG(2,q)$ \emph{determines} a direction $d$
if there are distinct points $P,Q\in U$ such that the line $PQ$ has
direction $d$. We write $D_U\subseteq\ell_\infty$ for the set of
directions determined by $U$, and $|D_U|$ for their number.
Affine equivalence means equivalence under an invertible
$\F$-affine transformation. If $U$ is a \emph{$q$-set}, that is, a set
of $q$ points, and $|D_U|<q+1$, then choosing a direction not determined by $U$
as $\infty$ represents $U$ as the graph of a function $\F\to\F$; thus
the point-set and function formulations are interchangeable.

Put
\[
\UU_q=\{(x,x^{(q+1)/2}):x\in\F\}.
\]
We call $x^{(q+1)/2}$ the \emph{absolute value function} over $\F$, since
\[
x^{(q+1)/2}=
\begin{cases}
x,&\text{if $x$ is a square in $\F$ (including zero)},\\
-x,&\text{if $x$ is a nonsquare in $\F$}.
\end{cases}
\]
When $q\equiv3\pmod4$, squares and nonsquares play the roles of positive
and negative numbers. When $q\equiv1\pmod4$, the analogy is only formal:
if $x$ is a nonsquare, then $-x$ is still a nonsquare.
The set $\UU_q$ determines exactly $(q+3)/2$ directions
(Lemma~\ref{lem:two-lines}).
An \emph{affine $\Fp$-linear set} of size $q$ is a coset of an
$n$-dimensional $\Fp$-subspace of $\F^2$.

\Needspace{6\baselineskip}
\begin{theorem}\label{thm:main}
Let $q=p^n$, where $p>2$ is prime and $n\ge1$. If
$U\subseteq\AG(2,q)$ has $q$ points and
\[
|D_U|=\frac{q+3}{2},
\]
then $U$ is affinely equivalent to $\UU_q$.
\end{theorem}

\begin{theorem}\label{thm:stability}
Let $q=p^n$, where $p>2$ is prime and $n\ge2$. If
$U\subseteq\AG(2,q)$ has $q$ points and
\[
|D_U|\le\frac{q+p}{2},
\]
then either $U$ is an affine $\Fp$-linear set, or it is affinely
equivalent to $\UU_q$. Consequently no $q$-set determines $N$ directions
with $(q+3)/2<N\le(q+p)/2$.
\end{theorem}

For $n=2$, these conclusions were already known.
G\'acs--Lov\'asz--Sz\H onyi \cite[Theorem~1.3]{GLS} prove that, for
$q=p^2$ and $|D_U|\ge(q+3)/2$, either $U$ is affinely equivalent to
$\UU_q$ or $|D_U|\ge(q+p)/2+1$. Together with the established
linearity theorem for $|D_U|<(q+3)/2$ \cite{BBBS,Ball}, this implies
both theorems above when $n=2$. Their proof leaves the remaining
case $p=5$ to a computer check \cite[p.~200]{GLS}.
The characteristic-five argument below gives a computer-free proof,
including $q=25$, using the elementary polynomial comparison of line
intersection numbers in Lemma~\ref{lem:index-comparison}.
For $q=p^2$, the bound $(q+p)/2$ in Theorem~\ref{thm:stability} is
sharp: the example of Polverino, Sz\H onyi, and Weiner \cite{PSW}
determines $(q+p)/2+1$ directions \cite{GLS}. It is neither affine
$\Fp$-linear, since such sets determine at most $(q-1)/(p-1)$
directions, nor affinely equivalent to $\UU_q$.

For a graph, the first alternative says that its reduced graph polynomial has the form
$c+\sum_{i=0}^{n-1}c_iX^{p^i}$. When $p=3$, the hypothesis in
Theorem~\ref{thm:stability} is $|D_U|\le(q+3)/2$: the strict inequality gives
the affine $\Fp$-linear alternative, and equality is classified by
Theorem~\ref{thm:main}. For $p>3$, equality $|D_U|=(q+p)/2$ is
excluded by Proposition~\ref{prop:equality}.

\subsection*{Historical background}
The polynomial method for this problem originates in R\'edei's theory
of lacunary polynomials. The prime-field lower bound $(p+3)/2$ is due
to R\'edei and Megyesi \cite{Redei}. The characterization of the
equality case also follows implicitly from R\'edei's results
\cite{Redei}; Lov\'asz and Schrijver stated and proved it explicitly
\cite{LS}. Dress, Klin, and Muzychuk independently obtained
the prime-field bound in their work on $p$-configurations \cite{DKM}.
G\'acs established a further gap above $(p+3)/2$ over prime fields
\cite{Gacs}.

The extension-field direction theory was developed by Blokhuis,
Brouwer, and Sz\H onyi \cite{BBS}, then by Blokhuis, Ball, Brouwer,
Storme, and Sz\H onyi \cite{BBBS}.
Ball completed the cases $s=2,3$ left open in \cite{BBBS}, using
Hasse derivatives \cite{Ball}: the lower bound $|D_U|\ge q/2+1$ for
$s=2$, and $\mathbb F_3$-linearity for $s=3$. Here $s$ is the largest
power of the characteristic dividing $|U \cap \ell|$ for every line
$\ell$ with direction in $D_U$. The $s=3$ result supplies the additional
linearity case needed here when $p=3$.
G\'acs, Lov\'asz, and Sz\H onyi \cite{GLS} developed a polynomial
reduction for proper powers $q$ of $p$ and $(q+3)/2\le|D_U|\le(q+p)/2$,
and obtained the classification and gap when $q=p^2$.

Our earlier structural results for bisecants of R\'edei-type blocking
sets \cite{Csajbok}, in the form of Lemma~\ref{lem:index-comparison},
supply the final steps in characteristics five and three
(Sections~\ref{sec:five} and~\ref{sec:three-reflections}) and in the
exclusion of $|D_U|=(q+p)/2$ when $p>3$ (Proposition~\ref{prop:equality}).
We include the short polynomial argument from our lecture
notes \cite[Theorem~1.4.7]{CsajbokNotes}.
For a broader account of direction problems, see De Beule's
survey \cite{DeBeule}. In even characteristic, Sziklai's recent
preprint treats the remaining case $|D_U|=q/2+1$ \cite{Sziklai26}.

\subsection*{Outline of the proof}
Let $n\ge2$ and $(q+3)/2\le|D_U|\le(q+p)/2$. The starting point is the
lacunary form of the R\'edei polynomial (Lemma~\ref{lem:gap}).
Specializing it to a determined direction, the intersection
identities of Lemma~\ref{lem:special} show that every nonempty line
intersection has size prime to $p$. Its \emph{centroid} is therefore
well-defined as the average of its points in $\F^2$. For each determined
direction $d$, the centroids of the nonempty intersections with lines of
direction $d$ lie on a line $L_d$, called the \emph{centroid line} of $d$.
All these lines are concurrent (Lemma~\ref{lem:centroid}).
Bounds on the numbers of points of tangency in these directions,
combined with the coincidence results Corollary~\ref{cor:coincidence}
and Lemma~\ref{lem:boundary-coincidence}, yield affine symmetries of
$U$ (Lemmas~\ref{lem:type-two-reflections} and~\ref{three:lem:reflection}).
A large secant then forces $U$ onto two lines
(Lemma~\ref{lem:large-secant}). The remaining configurations are
excluded by the polynomial comparison in Lemma~\ref{lem:index-comparison}
and a quadratic equation restricting which directions can have a
prescribed centroid line in characteristic five
(Proposition~\ref{prop:five}). Lemma~\ref{lem:reflection-secant} gives the common
counting argument used in characteristic three and at $|D_U|=(q+p)/2$
(Section~\ref{sec:three-reflections} and Proposition~\ref{prop:equality}).

The external inputs are the coefficient bounds and the polynomial
reduction of G\'acs, Lov\'asz, and Sz\H onyi
\cite[Lemma~3.3, Corollary~3.7(i), Theorem~3.8]{GLS}, which enter
through Lemmas~\ref{lem:gap} and~\ref{lem:coincidence}, and, for
$|D_U|<(q+3)/2$ in Theorem~\ref{thm:stability}, the linearity theorem
\cite{BBBS,Ball}. Otherwise the arguments use only polynomial
divisibility, ordinary derivatives, linear algebra, and counting.

\section{A common polynomial framework}\label{sec:gap}
The prime-field case of Theorem~\ref{thm:main} follows from
Lov\'asz and Schrijver \cite{LS}. Throughout Sections~\ref{sec:gap}--\ref{sec:three-reflections}, assume
$n\ge2$ and work in the range
\begin{equation}\label{eq:range}
\frac{q+3}{2}\le |D_U|\le\frac{q+p}{2}.
\end{equation}
Put
\begin{equation}\label{eq:notation}
M=\frac{q+1}{2},\qquad m=M-1=\frac{q-1}{2},\qquad K=\frac{q-1}{p-1},
\end{equation}
and define $\delta$ by $|D_U|=M+1+\delta$.
Thus $\delta=0$ when $|D_U|=(q+3)/2$; otherwise $p>3$ and
$0<\delta\le(p-3)/2$. Cardinalities and inequalities are interpreted
in the integers, while coefficients and derivatives are taken in
$\F$. In particular, since $q=0$ in $\F$, the integers $M$ and $m$, viewed as
elements of $\F$, are $M=1/2$ and $m=-1/2$. For a polynomial $g$,
we write $[X^k]g$ for the coefficient of $X^k$ in $g$, and use the
same notation for any other variable. Subscripts $X$ and $Y$ denote
partial derivatives, for example $R_X=\partial R/\partial X$, and a
prime denotes the derivative of a polynomial in one variable. For a
nonzero polynomial $g$ and $u\in\F$, $\ord_u g$ denotes the
multiplicity of $u$ as a root of $g$, so $\ord_u g=0$ if $g(u)\ne0$.

The \emph{intersection number} of an affine line $\ell$ with $U$ is
$|U\cap\ell|$. A \emph{tangent} to $U$ has intersection number one,
and a $k$-\emph{secant} has intersection number $k\ge2$.
For a fixed determined direction, the points of tangency are called
\emph{singleton points}. A \emph{singleton direction} at $P\in U$
is a determined direction whose line through $P$ is tangent to $U$.

For a finite set $S$ of points with $p\nmid|S|$, its \emph{centroid}
is the point
\[
\frac1{|S|}\sum_{P\in S}P\in\F^2.
\]
By Lemma~\ref{lem:special} below, $p\nmid|U\cap\ell|$ for every line
$\ell$ in a determined direction meeting $U$, and lines whose
directions are not determined by $U$ meet $U$ in one point. So every nonempty intersection
$U\cap\ell$ has a centroid.

Since $|D_U|<q+1$, there are affine coordinates $(x,z)$ in which the
vertical direction is not determined by $U$; we call them \emph{graph coordinates}
for $U$. In such coordinates $U=\{(x,f(x)):x\in\F\}$ for a function
$f:\F\to\F$. We identify $f$ with its unique polynomial representative
of degree at most $q-1$, the \emph{reduced graph polynomial} of $U$ in
these coordinates, and set
\begin{equation}\label{eq:redei}
R(X,Y)=\prod_{x\in\F}(X+xY+f(x))
       =X^q+\sum_{i=1}^q h_i(Y)X^{q-i}.
\end{equation}
For $y\in\F$, the lines $yx+z=-t$, $t\in\F$, form the parallel class
of slope $-y$, and the multiplicity of $t$ as a root of $R(X,y)$ is the
number of points of $U$ on the line $yx+z=-t$.
Since the vertical direction is not determined by $U$, we have
$D_U\subseteq\F$ under the identification of directions with slopes;
thus $-y\in D_U$ if and only if $f(X)+yX$ is not a permutation
polynomial of $\F$. Put
\[
T=\{y\in\F:-y\notin D_U\}.
\]
Then $|T|=q-|D_U|$, and for $y\in T$ every line $yx+z=-t$ contains
exactly one point of $U$, so the values $yx+f(x)$ run through $\F$.
Hence
\begin{equation}\label{eq:T}
|T|=m-1-\delta,\qquad R(X,y)=X^q-X\quad(y\in T).
\end{equation}

\Needspace{12\baselineskip}
\begin{lemma}[{\cite[Lemma~3.3, Corollary~3.7(i), Theorem~3.8]{GLS}}]\label{lem:gap}
In every choice of graph coordinates, the reduced graph polynomial has
degree at most $|D_U|-1$. After a possible affine change of graph
coordinates when $\delta>0$,
\begin{equation}\label{eq:gap}
R=X^q+a(Y)X^M+a_1(Y)X^{M-1}
  +\text{terms of $X$-degree at most $M-2$},
\end{equation}
where $\deg a=m-1$, $\deg a_1\le m$, and both $a$ and $a_1$ vanish
on $T$.
When $|D_U|=(q+3)/2$, the form \eqref{eq:gap} holds with $\deg a=m-1$ in every
choice of graph coordinates, and
\begin{equation}\label{eq:aE}
a(Y)=\lambda\prod_{y\in T}(Y-y),\qquad\lambda\in\F^*.
\end{equation}
\end{lemma}

\begin{proof}
Our $R$ is the R\'edei polynomial of \cite[Section~3]{GLS}, and
$I_1(f)$ denotes the least index $i\ge1$ with $h_i\not\equiv0$.
In every graph coordinate system, \cite[Lemma~3.3(ii)]{GLS} gives
$I_1(f)\ge q+1-|D_U|=m-\delta$, and \cite[Corollary~3.7(i)]{GLS}
gives $\deg f\le q-I_1(f)\le|D_U|-1$. Moreover, $|D_U|\ge m+2>K$, so
\cite[Lemma~3.3(i)]{GLS} gives $I_1(f)\le m$.

Next, $\deg h_i\le i-1$ and $h_i|_T=0$ for $1\le i\le q-2$:
the coefficient of $Y^i$ in $h_i$ is the $i$th elementary symmetric
function of $\F$, which vanishes, and $R(X,y)=X^q-X$ for $y\in T$.

Suppose first that $\delta=0$. Then $I_1(f)=m$ in every graph
coordinate system, so \eqref{eq:gap} holds with $a=h_m\ne0$. Since
$\deg h_m\le m-1$ and $h_m$ vanishes at the $m-1$ elements of $T$,
this gives \eqref{eq:aE} and $\deg a=m-1$.

For $\delta>0$, \cite[Theorem~3.8]{GLS} gives graph coordinates with
$I_1(f)=m$ and $\deg f=M$, so \eqref{eq:gap} holds with $a=h_m\ne0$.
Write $f(X)=\sum_{j=0}^{q-1}f_jX^j$. Expanding the product defining
$R$, and using $e_{m-1}(\F\setminus\{x\})=(-x)^{m-1}$ for the
elementary symmetric function $e_{m-1}$, we obtain
\[
[Y^{m-1}]h_m(Y)=\sum_{x\in\F}f(x)(-x)^{m-1}=(-1)^mf_M\ne0,
\]
so $\deg a=m-1$. In both cases $a=h_m$ and $a_1=h_{m+1}$, so the
bounds above give $\deg a_1\le m$ and show that $a$ and $a_1$ vanish
on $T$.
\end{proof}

From now on we fix graph coordinates in which \eqref{eq:gap} holds
with $\deg a=m-1$: for $\delta=0$ any choice will do, and for
$\delta>0$ we take those supplied by \cite[Theorem~3.8]{GLS}.

\begin{lemma}[a consequence of {\cite[Lemma~3.3(ii), Corollary~3.7(i)]{GLS}}]\label{lem:coincidence}
Let $U,V$ be $q$-sets, each determining at least two directions, such
that some direction is determined by neither of them. If $|U\cap V|\ge\max\{|D_U|,|D_V|\}$,
then $U=V$.
\end{lemma}
\begin{proof}
Make such a direction vertical. By the cited results, as in the proof
of Lemma~\ref{lem:gap}, the reduced graph polynomials of $U$ and $V$
have degrees at most $|D_U|-1$ and $|D_V|-1$, respectively. Their
difference vanishes at the $x$-coordinates of the points of $U\cap V$,
which are distinct. It therefore has more roots than its degree, and
hence is zero.
\end{proof}

\begin{corollary}\label{cor:coincidence}
If two $q$-sets each determine $(q+3)/2$ directions and have more than
$M=(q+1)/2$ common points, then they are equal.
\end{corollary}
\begin{proof}
Let the sets be $U,V$. Since $|D_U|=M+1>1$, the set $U$ is not
collinear. If a line $\ell$ contains $k\ge2$ points of $U$, choose
$P\in U\setminus\ell$. The lines joining $P$ to these $k$ points
have distinct directions, all different from the direction of $\ell$.
Thus $k+1\le |D_U|=M+1$, so no line contains more than $M$ points of $U$.
It follows that $U\cap V$ is not collinear, so
$|D_U\cap D_V|\ge3$. Hence
\[
|D_U\cup D_V|=|D_U|+|D_V|-|D_U\cap D_V|\le q,
\]
so some direction is determined by neither set.
Since $|U\cap V|\ge M+1=|D_U|=|D_V|$, Lemma~\ref{lem:coincidence}
gives $U=V$.
\end{proof}

\begin{remark*}
The threshold in Corollary~\ref{cor:coincidence} is sharp. Let $H_2$ be
the set of nonzero squares in $\F$, and let $\nu$ be a nonsquare. Put
\[
\begin{aligned}
U&=\{(0,0)\}\cup(H_2\times\{0\})\cup(\{0\}\times H_2),\\
V&=\{(0,0)\}\cup(H_2\times\{0\})\cup(\{0\}\times\nu H_2).
\end{aligned}
\]
These are distinct $q$-sets. Their determined directions are the two
axis directions together with the slopes in $-H_2$ and $-\nu H_2$,
respectively. Thus each determines $(q+3)/2$ directions, whereas
$U\cap V=\{(0,0)\}\cup(H_2\times\{0\})$ has exactly $M$ points.
\end{remark*}

\Needspace{6\baselineskip}
\begin{lemma}\label{lem:nonexceptional}
For every $y\in\F$ with $-y\in D_U$, one has $a(y)\ne0$.
\end{lemma}
\begin{proof}
For $\delta=0$, this follows from \eqref{eq:aE}. Assume $\delta>0$.
Suppose $a(y)=0$ for some $y$ with $-y\in D_U$, and write
$P=R(X,y)=X^q+g$, where $\deg g\le m$. Since $-y\in D_U$,
$g+X\ne0$. All distinct roots of $P$ are roots of $g+X$, so their
number $k$ is at most $m$. If $P'\ne0$, then
$q-k\le\deg\gcd(P,P')\le m-1$, a contradiction. Hence $P'=0$, so the distinct roots
$t_1,\ldots,t_k$ of $P$ have multiplicities $w_i$ divisible
by $p$, and $k\le q/p$.
For $1\le j\le\delta<p$, put
\[
R_j(X)=\frac1{j!}\frac{\partial^jR}{\partial Y^j}(X,y).
\]
Since $j<p$, $j!\ne0$ in $\F$. By \eqref{eq:gap}, $\deg R_j\le M$.
By the product rule, differentiating $j$ of the linear factors removes
at most $j$ factors vanishing at any $t_i$. Thus
\[
\prod_{i=1}^k(X-t_i)^{w_i-j}\mid R_j.
\]
Here $j<p\le w_i$. If $R_j\ne0$, then
\[
\deg R_j\ge q-jk\ge q-\frac{jq}{p}
 \ge\frac{q(p+3)}{2p}>M,
\]
a contradiction. Therefore $R_j=0$, and its $X^M$ coefficient gives
$a^{(j)}(y)/j!=0$. Write $a(Y)=\sum_{i\ge0}b_i(Y-y)^i$. Then
$b_0=a(y)=0$, and $b_j=a^{(j)}(y)/j!=0$ for $1\le j\le\delta$.
Hence $(Y-y)^{\delta+1}$ divides $a$. By Lemma~\ref{lem:gap}, $a$
also vanishes on $T$, which does not contain $y$, so
$\deg a\ge |T|+\delta+1=m$, impossible.
\end{proof}

\begin{lemma}\label{lem:quotient}
The polynomial $a$ divides $a_1$. Consequently $a_1/a$ is affine in $Y$.
\end{lemma}
\begin{proof}
When $|D_U|=(q+3)/2$, both $a$ and $a_1$ vanish on $T$ by
Lemma~\ref{lem:gap}, and
\eqref{eq:aE} proves the assertion directly. Suppose $\delta>0$;
then $q\ge p^2$. First establish the bivariate identity
\begin{equation}\label{eq:bivariate}
\bigl(R(X,Y)-X^q+X\bigr)R_X(X,Y)-(Y^q-Y)R_Y(X,Y)=\frac{a(Y)^2}{2}R(X,Y).
\end{equation}
Modulo a factor $X+xY+f(x)$ of $R(X,Y)$, we have $R\equiv0$ and
$X\equiv-xY-f(x)$, so $X^q-X\equiv-x(Y^q-Y)$, as $x,f(x)\in\F$. Hence
the first parenthesis is congruent to $x(Y^q-Y)$. Also
$R_Y\equiv xR_X$, since both sides are congruent to $x$ times the
product of the other factors of $R(X,Y)$. Thus this factor divides the
left side. The $q$ factors of $R(X,Y)$ are pairwise coprime, hence
$R(X,Y)$ divides the left side. Since $q=0$ in $\F$, \eqref{eq:gap}
gives $R_X(X,Y)=Ma(Y)X^{M-1}+\cdots$, so the left side has $X$-degree
at most $2M-1=q$ and $X^q$-coefficient $Ma(Y)^2=a(Y)^2/2$. As
$R(X,Y)$ is monic in $X$, this gives \eqref{eq:bivariate}.
Put $A(X,Y)=R(X,Y)-X^q+X$ and $c(Y)=a(Y)^2/2$. Then
\begin{equation}\label{eq:bivariateA}
A(X,Y)A_X(X,Y)=\bigl(1+c(Y)\bigr)A(X,Y)+c(Y)(X^q-X)+(Y^q-Y)A_Y(X,Y).
\end{equation}
Write $A(X,Y)=\sum_{j=0}^{M}a_j(Y)X^{M-j}$, where $a_0(Y)=a(Y)$, and set
$\kappa=(p+1)/2<m$. Except for $c(Y)X^q$, the right side has
$X$-degree at most $M$. For $2\le j\le\kappa$, comparison at
$X^{q-j}$, using $2AA_X=(A^2)_X$, gives
$[X^{q+1-j}]A^2=0$, since
$q+1-j=1-j$ is nonzero. Thus the coefficients satisfy the finite
recurrence
\begin{equation}\label{eq:coefficientrecurrence}
2aa_j=-\sum_{i=1}^{j-1}a_i a_{j-i}
\qquad(2\le j\le\kappa).
\end{equation}
Put $\rho=a_1/a\in\F(Y)$ and $\alpha_j=a_j/a$. Dividing
\eqref{eq:coefficientrecurrence} by $a^2$ gives
$2\alpha_j=-\sum_{i=1}^{j-1}\alpha_i\alpha_{j-i}$ for
$2\le j\le\kappa$, with $\alpha_0=1$ and $\alpha_1=\rho$. Hence $\rho$
determines $\alpha_2,\ldots,\alpha_\kappa$ uniquely. Now let
$C(Z)=(1+2\rho Z)^\kappa=\sum_jc_jZ^j\in\F(Y)[Z]$. Then $c_0=1$ and
$c_1=2\kappa\rho=\rho$, since $2\kappa=p+1=1$ in $\F$. Moreover,
\[
C(Z)^2=(1+2\rho Z)^{p+1}
       =(1+2\rho Z)\bigl(1+(2\rho Z)^p\bigr)
       \equiv1+2\rho Z\pmod{Z^{\kappa+1}},
\]
so comparing the coefficients of $Z^j$ for $2\le j\le\kappa$ shows that
the $c_j$ satisfy the same recurrence. Therefore $\alpha_j=c_j$, that is,
\[
a_j=2^j\binom\kappa j\frac{a_1^j}{a^{j-1}}
\qquad(1\le j\le\kappa).
\]
At $j=\kappa$ this yields
$a^{\kappa-1}a_\kappa=2^\kappa a_1^\kappa$, hence
$a^{\kappa-1}\mid a_1^\kappa$.
If $a_1=0$ there is nothing to prove. Otherwise let an irreducible
factor have multiplicity $\eta$ in $a$ and $\eta_1$ in $a_1$.
Dividing $a$ by the $|T|=m-1-\delta$ distinct linear factors
corresponding to $T$ gives a quotient of degree $\delta$.
Hence $\eta\le\delta+1\le\kappa-1$. Divisibility gives
$\kappa\eta_1\ge(\kappa-1)\eta$. Since $\eta<\kappa$, integrality forces
$\eta_1\ge\eta$. Hence $a\mid a_1$, and the degree bounds in
Lemma~\ref{lem:gap} finish the proof.
\end{proof}

\section{Intersection identities and concurrent centroid lines}
\begin{lemma}\label{lem:special}
Fix $y\in\F$ such that the direction of slope $-y$ is determined, and write
\[
P(X)=R(X,y),\qquad A(X)=P(X)-X^q+X,\qquad
c=\frac{a(y)^2}{2}.
\]
Exactly $M$ lines with this direction meet $U$. The polynomial $A$
is squarefree, has the parameters of these $M$ lines as its distinct
roots, and satisfies
\begin{align}
AP'&=cP,\label{eq:AP}\\
AA'&=(1+c)A+c(X^q-X).\label{eq:AA}
\end{align}
Moreover,
\begin{equation}\label{eq:types}
c\in\{2,1/2\}.
\end{equation}
For every line $\ell$ in this direction meeting $U$, the number
$|U\cap\ell|$ is not divisible by $p$. If $w=|U\cap\ell|>1$,
then $w\equiv c\pmod p$.
\end{lemma}

\begin{proof}
All roots of $P$ lie in $\F$, so they are simple roots of $X^q-X$.
At a root of $P$ of multiplicity $w$, the polynomial $(X^q-X)P'$
therefore vanishes to order at least $1+(w-1)=w$. Hence
$P\mid(X^q-X)P'$, and since $AP'=PP'-(X^q-X)P'$, also $P\mid AP'$.
The polynomial $A$ has degree $M$ and leading coefficient
$a(y)\ne0$, by Lemma~\ref{lem:nonexceptional}, and $P'=A'-1$ has
degree $M-1=m$ and leading coefficient $Ma(y)$, because $M=1/2\ne0$ in
$\F$. Thus $\deg(AP')=q$, and comparing leading coefficients gives
$AP'=Ma(y)^2P=cP$. This proves \eqref{eq:AP}; substituting $P=A+X^q-X$
and $P'=A'-1$ gives \eqref{eq:AA}.

At a root $t$ of $P$ of multiplicity $w$, the polynomial $A$ vanishes,
since $t^q=t$, and $P'$ vanishes to order at least $w-1$. As
$c=a(y)^2/2\ne0$, \eqref{eq:AP} forces both orders to be exact. Thus $A$
has a simple root at $t$, and $p\nmid w$, since otherwise $P'$ would
vanish to order at least $w$. Conversely, every root of $A$ is a root
of $P$ by \eqref{eq:AP}. Hence $A=a(y)\prod_t(X-t)$, where $t$ runs over
the distinct roots of $P$. In particular, $A$ is squarefree, and $P$
has exactly $\deg A=M$ distinct roots, the parameters of the lines in
this direction that meet $U$.

At a root $t$ of $P$ of multiplicity $w$, write
\[
P=(X-t)^wP_t,\qquad A=(X-t)A_t,
\qquad P_t(t)\ne0,\quad A_t(t)=A'(t).
\]
Substitute these factorizations into \eqref{eq:AP}, cancel
$(X-t)^w$, and set $X=t$. The product rule gives
$wA_t(t)P_t(t)=cP_t(t)$, hence
\begin{equation}\label{eq:rootderivative}
wA'(t)=c.
\end{equation}
If $w=1$, this says $A'(t)=c$; if $w>1$, then $P'(t)=0$ and
$A'(t)=1$. Thus every secant in this direction has intersection number
congruent to $c$.

Differentiate \eqref{eq:AA} twice, one derivative at a time, to obtain
\begin{align}
AA''+(A'-1)(A'-c)&=0,\label{eq:diffone}\\
AA'''+(3A'-1-c)A''&=0.\label{eq:difftwo}
\end{align}
At a root of $A$, the coefficient of $A''$ in \eqref{eq:difftwo} is
$2-c$ if $A'=1$, and $2c-1$ if $A'=c$. If $c$ were neither $2$ nor
$1/2$, it would follow that $A''$ vanishes at all $M$ roots of $A$.
But $A''$ has degree $M-2$ and leading coefficient
$a(y)M(M-1)=-a(y)/4\ne0$. This contradiction proves \eqref{eq:types}.
\end{proof}

When $p>3$, we call a determined direction \emph{type $2$} or \emph{type $1/2$}
according to its value of $c$; the two values are distinct in this case.
Every secant in a direction of type $c$ has size congruent to $c$
modulo $p$, where $1/2\equiv(p+1)/2$. In particular, the type depends
only on $U$ and the direction.
Also every secant contains at most $M$ points of $U$, since each
of the other $M-1$ lines in its direction meeting $U$ contains at
least one point of $U$.

When $p=3$, the two values of $c$ coincide: $c=2=1/2=-1$.
Thus Lemma~\ref{lem:special} gives
\begin{equation}\label{eq:three-special}
AP'=-P,\qquad AA'=X-X^q,
\end{equation}
and every secant has intersection number $2$ modulo $3$.

\begin{lemma}\label{lem:centroid}
For each determined direction $d$, the centroids of the nonempty
intersections $U\cap\ell$, where $\ell$ has direction $d$, lie on
an affine line $L_d$, the \emph{centroid line} of $d$, whose direction
is not $d$. In particular, all singleton
points in direction $d$ lie on $L_d$, as does the midpoint of the
two points of $U$ on each bisecant in that direction. All the lines
$L_d$ pass through a common affine point $O$. We do not yet claim that
$O\in U$.
\end{lemma}

\begin{proof}
Put
\[
\beta(Y):=\frac{a_1(Y)}{M a(Y)}.
\]
By Lemma~\ref{lem:quotient}, $\beta(Y)$ is affine; write
$\beta(Y)=\beta_0Y+\beta_1$. Translate $U$ by $-(\beta_0,\beta_1)$.
Translations preserve directions and commute with taking centroids, so
it suffices to prove the lemma for the translated set. Its R\'edei
polynomial is $R(X-\beta(Y),Y)$. Since
$(X-\beta(Y))^q=X^q-\beta(Y)^q$, this polynomial again has the form
\eqref{eq:gap} with the same $a(Y)$, so
Lemmas~\ref{lem:nonexceptional} and~\ref{lem:special} apply to it with
the same values of $c$. Its coefficient of $X^{M-1}$ is
$a_1(Y)-Ma(Y)\beta(Y)$, which vanishes by the definition of
$\beta(Y)$. Replacing $U$ by the translated set, we may therefore
assume $a_1(Y)=0$.

Fix $y\in\F$ with $-y\in D_U$ and retain $P(X)$, $A(X)$, and $c$ from
Lemma~\ref{lem:special}. Put $H(X)=R_Y(X,y)$. From now on, all
polynomials are in the variable $X$. At a root $t$ of $P$
of multiplicity $w$, differentiation of the product defining $R(X,Y)$
with respect to $Y$ shows
that $H$ has multiplicity at least $w-1$. Since $A$ is the radical of
$P$ up to a scalar, $P$ divides $AH$. The coefficients of
$X^{q-1},\ldots,X^{M+1}$ in $R(X,Y)$ vanish identically in $Y$, so
their $Y$-derivatives vanish as well. Consequently $\deg H\le M$,
and $\deg(AH)\le2M=q+1$. Divisibility therefore gives
\begin{equation}\label{eq:H}
A(X)H(X)=(uX+v)P(X)
\end{equation}
for some $u,v\in\F$, allowing $u=0$. Since $a_1(Y)=0$, the
coefficients of $X^{M-1}$ in $A$ and $H$, namely $a_1(y)$ and
$a_1'(y)$, vanish. Hence $[X^q](AH)=0$, while the
coefficient of $X^q$ on the right is $v$, because $M+1<q$.
Thus $v=0$.

Since $AH=uXP$ and $AP'=cP$, we have $A(cH-uXP')=0$, and hence
$cH=uXP'$. For a nonempty intersection
$S_t=\{(x,z)\in U:yx+z=-t\}$, put $w=|S_t|$, let $(\bar x,\bar z)$
be its centroid, and write $P(X)=(X-t)^wP_t(X)$ with $P_t(t)\ne0$.
Differentiate the product defining $R(X,Y)$ with respect to $Y$ and
set $Y=y$. By the product rule, the term coming from the factor
$X+xY+z$ with $(x,z)\in S_t$ is $x(X-t)^{w-1}P_t(X)$, and every other
term is divisible by $(X-t)^w$. Hence, modulo $(X-t)^w$,
\[
H\equiv(X-t)^{w-1}P_t\sum_{(x,z)\in S_t}x,\qquad
P'\equiv w(X-t)^{w-1}P_t.
\]
Dividing $cH=uXP'$ by $(X-t)^{w-1}$ and setting $X=t$ gives
$c\sum_{(x,z)\in S_t}x=utw$, since $P_t(t)\ne0$. As $p\nmid w$,
this says $c\bar x=ut$.
Since $y\bar x+\bar z=-t$, this gives $(c+uy)\bar x+u\bar z=0$.
Hence the centroids of all nonempty intersections $S_t$ lie on
\[
L_d=\{(x,z)\in\F^2:(c+uy)x+uz=0\},
\]
where $d$ is the direction of slope $-y$. Since $c\ne0$, the
coefficients $c+uy$ and $u$ are not both zero, so $L_d$ is a line
through the origin. Its direction is not $d$, because
$(c+uy)\cdot1-u\cdot y=c\ne0$. Translating back proves the assertion,
with $O=(\beta_0,\beta_1)$.
\end{proof}

For each determined direction $d$, the line $L_d$ does not have
direction $d$, so for every point $P$ the line through $P$ with
direction $d$ meets $L_d$ in a unique point $Q_P$. Define the
\emph{centroid reflection} $r_d$ by $r_d(P)=2Q_P-P$. It is an affine
involution that fixes $L_d$ pointwise and maps every line with
direction $d$ to itself. After translating $O$ to the origin, $r_d$ is
linear, with eigenspaces $L_d$, for eigenvalue $1$, and the line
through $O$ with direction $d$, for eigenvalue $-1$.

\begin{lemma}\label{lem:reflection-axes}
If $r_d(U)=U$ and $r_e(U)=U$ for distinct determined directions
$d,e$, then $L_d\ne L_e$.
\end{lemma}
\begin{proof}
Suppose $L_d=L_e=L$. Translate $O$ to the origin and choose a basis
with first vector on $L$. The reflections then have matrices
\[
r_d=\begin{pmatrix}1&\theta\\0&-1\end{pmatrix},\qquad
r_e=\begin{pmatrix}1&\theta'\\0&-1\end{pmatrix},\qquad
\theta\ne\theta'.
\]
Their product sends $(x,z)$ to $(x+(\theta'-\theta)z,z)$.
Its fixed points are exactly the points of $L$, and every other
orbit has length $p$. Since it preserves $U$,
$|U\cap L|\equiv q\equiv0\pmod p$.
On the other hand, some line in direction $d$ is tangent to $U$:
otherwise the $M$ nonempty intersections would contain at least
$2M=q+1$ points. Its point of tangency lies on $L_d=L$, so
$|U\cap L|>0$. Every nonempty line intersection has size prime to
$p$, by Lemma~\ref{lem:special} in a determined direction and by
size one otherwise. This is a contradiction.
\end{proof}

\begin{lemma}\label{lem:large-secant}
If a secant contains
$k>(q+3)/4$ points of $U$, then $k=M$ and $U$ lies on two
nonparallel lines.
\end{lemma}
\begin{proof}
Use the notation of Lemma~\ref{lem:special} in the direction of
the secant and translate its line parameter to zero. This preserves
\eqref{eq:AA}, since the translation belongs to $\F$. Write $A=X+Z$.
Then $\deg Z=M$ and $\ord_0 Z=k$, since $Z=P-X^q$ and $k\le M<q$.
Equation~\eqref{eq:AA} becomes
\[
XZ'-cZ+ZZ'=cX^q.
\]
The product $ZZ'$ has no term of degree less than $2k-1>M$,
whereas $\deg(XZ'-cZ)\le M$. Comparing terms of degree at most
$M$ gives $XZ'=cZ$, and hence $ZZ'=cX^q$. Multiplying the latter
identity by $X$ yields $cZ^2=cX^{q+1}$. Since $c\ne0$ and $2M=q+1$,
we obtain $Z=\pm X^M$ and $k=M$.
The other $M-1$ lines in this direction meeting $U$ are tangents.
Their points of tangency lie on the centroid line of this direction,
which is not parallel to the secant, by Lemma~\ref{lem:centroid}.
Hence the secant and this centroid line are two nonparallel lines
containing $U$.
\end{proof}

For a point $P$ and a direction $d$, write $Pd$ for the line through
$P$ with direction $d$.

\begin{lemma}[cf.~{\cite[Theorem~1.4.7]{CsajbokNotes}} and {\cite[Theorem~2.4(2)]{Csajbok}}]
\label{lem:index-comparison}
Suppose that distinct points $P,Q\in U$
each have exactly one singleton direction. If $|PQ\cap U|\ge3$,
then $|Pd\cap U|=|Qd\cap U|$ for every direction $d$.
In particular, their singleton directions agree and $PQ$ passes through $O$.
\end{lemma}
\begin{proof}
Choose a direction not determined by $U$ as $\infty$ and write
$U=\{(x,f(x)):x\in\F\}$, $P=(u,f(u))$, and $Q=(v,f(v))$.
Define
\[
F_P(W)=\prod_{x\in\F\setminus\{u\}}
\bigl((x-u)W-(f(x)-f(u))\bigr),
\]
and define $F_Q$ similarly. Both leading coefficients equal
$\prod_{\xi\in\F^*}\xi=-1$, so $\deg(F_P-F_Q)\le q-2$.
At each slope $t\in\F\setminus D_U$, the factors in either product
run through $\F^*$, giving $F_P(t)=F_Q(t)=-1$.
At a slope $d\in D_U$, the root multiplicities in $F_P$ and $F_Q$
are $|Pd\cap U|-1$ and $|Qd\cap U|-1$, respectively.
All but at most two slopes in $D_U$ are therefore common roots,
and the slope of $PQ$ is a common root of multiplicity at least two.
Unless $F_P-F_Q=0$, its roots have total multiplicity at least
\[
(q-|D_U|)+(|D_U|-2)+1=q-1,
\]
contrary to its degree bound. Hence $F_P=F_Q$, so the intersection
numbers agree in every finite direction, and both equal one in
direction $\infty$. Their unique singleton directions therefore
coincide, say at $e$. By Lemma~\ref{lem:centroid}, $P,Q\in L_e$,
so $PQ=L_e$ passes through $O$.
\end{proof}
The first assertion holds for every $q$-set with $|D_U|\le q$; only the
last sentence uses Lemma~\ref{lem:centroid}.

\begin{lemma}\label{lem:reflection-secant}
Suppose that $|D_U|\equiv0\pmod p$,
every secant meets $U$ in $2\pmod p$ points, and $r_d(U)=U$
for every $d\in D_U$. Then $U$ has a secant containing at least
$q-|D_U|+2$ points.
\end{lemma}
\begin{proof}
For $P\in U\setminus\{O\}$, let $\tau(P)$ count its singleton
directions. Each such direction $d$ satisfies $L_d=OP$.
Lemma~\ref{lem:reflection-axes} therefore gives $\tau(P)\le1$.
Counting the other points of $U$ on the lines through $P$ yields
\[
q-1\equiv |D_U|-\tau(P)\pmod p,
\]
so $\tau(P)=1$. Lemma~\ref{lem:index-comparison} now shows that
every secant avoiding $O$ is a bisecant.

Let $k$ be the largest secant size, and fix $P\in U\setminus\{O\}$.
Among the $|D_U|-1$ secants through $P$, only $OP$ can contain
more than two points of $U$. Hence
\[
q-1\le(k-1)+(|D_U|-2),
\]
which gives the assertion.
\end{proof}

The next lemma is a special case of the classification of blocking
sets of index three \cite{Cameron,Szonyi91,Senato}.

\begin{lemma}\label{lem:two-lines}
If $U$ lies on two nonparallel lines,
then $\delta=0$ and $U$ is affinely equivalent to $\UU_q$.
Conversely, $|D_{\UU_q}|=M+1$.
\end{lemma}
\begin{proof}
Both supporting lines are secants: otherwise one would contain at
least $q-1>M$ points. Make them the axes and write
\[
U=(A\times\{0\})\cup(\{0\}\times D)\cup U_0,
\qquad A,D\subseteq G:=\F^*,
\]
where $U_0$ is empty or $\{(0,0)\}$. If $U_0$ were empty, then
$|A|+|D|=q>|G|$, so every $r\in G$ would satisfy $rA\cap D\ne\varnothing$.
This would give $D/A=G$, so all $q+1$ directions would be determined,
a contradiction.
Thus $U_0=\{(0,0)\}$. The secant bound $M$ now gives
$|A|=|D|=m$, and $|D/A|=|D_U|-2=m+\delta$.
Choose a missing quotient $r$. The two half-sized sets $rA,D$ are
disjoint, so $D=G\setminus rA$. A quotient $s$ is missing precisely
when $sA=rA$. Hence the missing quotients form $rH$, where
$H=\Stab(A)$, and
\[
|H|=m-\delta>\frac{2m}{3}.
\]
For $\delta=0$ the inequality is immediate; otherwise it follows
from $q\ge p^2$ and $\delta\le(p-3)/2$. Since $H$ is proper, its
index in the group $G$ is exactly two. Therefore $\delta=0$, $H$ is
the subgroup $H_2$ of nonzero squares, and both $A,D$ are cosets of it.
Independent scaling of the axes gives
\begin{equation}\label{eq:axesmodel}
\{(0,0)\}\cup(H_2\times\{0\})\cup(\{0\}\times H_2).
\end{equation}
For a nonsquare $\nu$, the invertible map
$(u,v)\mapsto(u+\nu v,u-\nu v)$ sends this set to $\UU_q$:
$x^M=x$ on squares and $x^M=-x$ on nonsquares.
Conversely, \eqref{eq:axesmodel} determines the two axis directions
and the $m$ slopes in $-H_2$, proving the claimed count; see also
\cite[p.~190]{GLS}.
\end{proof}

\section{Tangent counts and centroid reflections when \texorpdfstring{$p>3$}{p>3}}
For this section assume $p>3$ and the range \eqref{eq:range}.
Let $h$ be the number of type $1/2$ directions.

\Needspace{6\baselineskip}
\begin{lemma}\label{lem:type-two-singletons}
If $d$ is a determined
direction of type $2$ and there are $s$ tangent lines to $U$ in
direction $d$, then
\[
1\le s\le K=\frac{q-1}{p-1}.
\]
Every secant in that direction has size at most $K+1$.
\end{lemma}
\begin{proof}
Use the notation of Lemma~\ref{lem:special}, with $c=2$.
The polynomial $A''$ is nonzero, with degree $M-2$ and leading
coefficient $-a(y)/4$. At a root $u$ corresponding to a tangent,
$A'(u)=2$ by \eqref{eq:rootderivative}. Put $\mu=\ord_u A''$.
Equation~\eqref{eq:difftwo} reads
\[
AA'''+3(A'-1)A''=0.
\]
Since $A(u+Z)=2Z+\text{terms of degree at least two}$,
comparison of the coefficient of $Z^\mu$ gives $2\mu+3=0$ in $\Fp$.
Thus $\mu\equiv(p-3)/2\pmod p$, and in particular
$\mu\ge(p-3)/2$.

For a secant with intersection number $w$, its root has multiplicity
exactly $w-2$ in $P''=A''$: indeed, $w\equiv2\pmod p$, so
$w(w-1)\ne0$ in $\F$. Exactly $M$ lines in direction $d$ meet
$U$, and their intersection numbers sum to $q$, whence
\begin{equation}\label{eq:boundary-singletoncount}
\sum_{\substack{\ell\text{ a secant}\\\mathrm{dir}(\ell)=d}}(|U\cap\ell|-2)=s-1.
\end{equation}
In particular $s\ge1$. Counting the roots of $A''$ just described,
with their multiplicities, yields
\[
M-2\ge\frac{p-3}{2}s+(s-1)
     =\frac{p-1}{2}s-1.
\]
This gives $s\le K$. Finally, each secant contributes
$w-2$ to the nonnegative sum in
\eqref{eq:boundary-singletoncount}, so $w\le s+1\le K+1$.
\end{proof}

\Needspace{6\baselineskip}
\begin{lemma}\label{lem:boundary-coincidence}
Let $p>3$, $q=p^n$, and $n\ge2$. Suppose that $U,V$ are $q$-sets
with $(q+3)/2\le|D_U|,|D_V|\le(q+p)/2$. If
\[
|U\cap V|>\max\Bigl\{\frac{q+p}{2}-1,\,(p-2)^2\Bigr\},
\]
then $U=V$.
\end{lemma}
\begin{proof}
Put $C=U\cap V$. By Lemma~\ref{lem:special}, no line contains
more than $M$ points of $U$. As $|C|>(q+p)/2-1\ge M$, the set $C$ is
not collinear. Write
$N_C=|D_C|$. A line containing $j\ge2$ points of $C$, together with
a point of $C$ outside that line, determines at least $j+1$
directions. Thus every line contains at most $N_C-1$ points of $C$.
Counting from a fixed point of $C$ gives
\[
|C|\le1+N_C(N_C-2)=(N_C-1)^2.
\]
Since $|C|>(p-2)^2$, it follows that $N_C\ge p$. Consequently
\[
|D_U\cup D_V|
\le |D_U|+|D_V|-|D_C|
\le q+p-p=q.
\]
Hence some direction is determined by neither $U$ nor $V$. Since
$|C|\ge(q+p)/2\ge\max\{|D_U|,|D_V|\}$, Lemma~\ref{lem:coincidence}
gives $U=V$.
\end{proof}

\begin{lemma}\label{lem:type-two-reflections}
For every type $2$ direction $d$, the centroid reflection $r_d$
preserves $U$.
Moreover, the map $d\mapsto L_d$ is injective on the type $2$
directions.
\end{lemma}
\begin{proof}
In direction $d$, let $s,b_2,t$ count tangents, bisecants, and
lines meeting $U$ in more than two points, respectively. Write the
latter intersection numbers as $2+pe_i$, where $e_i\ge1$, and put
$\varepsilon=\sum e_i$. Equation~\eqref{eq:boundary-singletoncount} gives
$s=1+p\varepsilon$, while $t\le\varepsilon$. The reflection fixes every singleton
point and interchanges the two points of $U$ on every bisecant, by
Lemma~\ref{lem:centroid}. These fixed points and pairs give
\begin{align}
|U\cap r_d(U)|
&\ge s+2b_2=q-p\varepsilon-2t\notag\\
&\ge q-(p+2)\varepsilon
 \ge q-\frac{p+2}{p}(K-1)=:C_0.\label{eq:type-two-overlap}
\end{align}
Here $\varepsilon=(s-1)/p\le(K-1)/p$ by
Lemma~\ref{lem:type-two-singletons}. Since
\[
C_0=q-\frac{p+2}{p(p-1)}(q-p),
\qquad \frac{p+2}{p(p-1)}<\frac12,
\]
we have $C_0>(q+p)/2$. Moreover, $C_0$ is increasing in $q$,
so $q\ge p^2$ gives
\[
C_0\ge p^2-p-2>(p-2)^2.
\]
The sets $U$ and $r_d(U)$ determine the same number of directions.
Lemma~\ref{lem:boundary-coincidence} now gives $r_d(U)=U$.

Injectivity now follows from Lemma~\ref{lem:reflection-axes}.
\end{proof}

\begin{lemma}\label{lem:half}
One has
\begin{equation}\label{eq:hcongruence}
3h\equiv6+4\delta\pmod p.
\end{equation}
Every type $1/2$ direction has at least
\[
B=M-K=\frac{q+1}{2}-\frac{q-1}{p-1}
\]
singleton points, all on its centroid line. If $|D_U|=(q+3)/2$,
$h\equiv2\pmod p$, so $h\ge2$.
\end{lemma}
\begin{proof}
The polynomial $a(Y)^2/2$ has degree $q-3$, so its sum over $\F$
is zero. It vanishes on $T$ by Lemma~\ref{lem:gap} and takes the values $2$ and $1/2$ at the
$|D_U|$ elements $y$ with $-y\in D_U$. Thus
\[
0=2(|D_U|-h)+h/2,\qquad 3h=4|D_U|=6+4\delta\quad\text{in }\Fp.
\]
In a type $1/2$ direction, each secant contains at least $(p+1)/2$
points of $U$. The total excess over one point on each of the $M$
lines meeting $U$ is $q-M=m$. Consequently there are at most
$2m/(p-1)=K$ secants and at least $M-K=B$ singleton points.
Lemma~\ref{lem:centroid} places them on a line.
\end{proof}

\begin{proposition}\label{prop:largep}
In the range \eqref{eq:range}, if $p\ge7$ and $h>0$, then
$\delta=0$ and $U$ is affinely equivalent to $\UU_q$.
\end{proposition}
\begin{proof}
A type $1/2$ direction has at least $B$ collinear singleton
points, by Lemma~\ref{lem:half}. Since
\[
B-\frac{q+3}{4}=\frac{(q-1)(p-5)}{4(p-1)}>0,
\]
Lemma~\ref{lem:large-secant} puts $U$ on two nonparallel lines.
Lemma~\ref{lem:two-lines} gives the assertion.
\end{proof}

\section{Characteristic five: maximal lines and centroid lines}\label{sec:five}
Assume $p=5$, $h>0$, and \eqref{eq:range}. Here $\delta\in\{0,1\}$.
Write
\begin{equation}\label{eq:five-notation}
q=4K+1,\qquad b=K+1,\qquad M=2K+1,\qquad |D_U|=2K+2+\delta.
\end{equation}
In the field, $K=1$, $b=2$, and $M=3$.

\begin{lemma}\label{lem:maxsize}
Either $U$ is contained in two nonparallel lines, or its largest
secant size is exactly $b$.
\end{lemma}
\begin{proof}
Let $k$ be the largest secant size. Lemma~\ref{lem:half} gives
$k\ge B=b=(q+3)/4$. If $k>b$, Lemma~\ref{lem:large-secant}
puts $U$ on two nonparallel lines. Otherwise $k=b$.
\end{proof}

Assume henceforth that the largest secant size is $b$.
A \emph{maximal line} will mean a $b$-secant. Translate the common
point $O$ of the centroid lines to the origin.

\Needspace{9\baselineskip}
\begin{lemma}\label{lem:residual-geometry}
Suppose that the largest secant size is $b$. Then the following hold.
\begin{enumerate}[label=(\roman*)]
\item The direction of every maximal line has exactly one
      $b$-secant, $K$ bisecants, and $K$ tangent lines to $U$.
      Its centroid line is another maximal line through $O$.
\item In every type $1/2$ direction there are exactly $K$ trisecants
      and $b$ tangent lines to $U$. Its centroid line is maximal.
\item Every maximal line passes through $O$, and $O\in U$.
      The \emph{centroid correspondence}
      $\ell\mapsto L_{\mathrm{dir}(\ell)}$, which sends a line to the
      centroid line of its direction, permutes the maximal lines.
\item There are exactly two or four maximal lines.
\end{enumerate}
\end{lemma}
\begin{proof}
Let $\ell$ be a maximal line. Its direction is of type $2$, since
$b\equiv2\pmod5$. In this direction, equation~\eqref{eq:boundary-singletoncount} and
Lemma~\ref{lem:type-two-singletons} give
\[
s\ge1+(b-2)=K,\qquad s\le K.
\]
Thus $s=K$, and the $b$-secant contributes the entire sum $s-1$;
every other secant in this direction is a bisecant. There are
$M-K-1=K$ such bisecants. The centroid line contains the $K$
singleton points. As $K\ge6$ and $K\equiv1\pmod5$, it cannot
contain exactly $K$ points of $U$, by Lemma~\ref{lem:special}.
It therefore contains $K+1=b$ points, proving (i).

For a type $1/2$ direction, write the intersection numbers of its
secants as $3+5e_i$ and put $\varepsilon=\sum e_i$. If there are $t$ secants,
counting the excess over one point on each of the $M$ lines meeting
$U$ gives
\[
2K=2t+5\varepsilon,\qquad s=M-t=b+\frac{5\varepsilon}{2}.
\]
All singleton points lie on a line, so $s\le b$. Hence $\varepsilon=0$,
$s=b$, and all $K$ secants in this direction are trisecants. This proves (ii).

By (i), the centroid correspondence maps the maximal lines to
maximal lines through $O$. Distinct maximal lines have distinct
directions, since (i) gives a unique maximal line in each such
direction. Lemma~\ref{lem:type-two-reflections} therefore makes the
correspondence injective. As the set of maximal lines is finite,
the correspondence permutes it, and every maximal line passes through
$O$. There are at least two, since the centroid line of a direction
$d$ does not have direction $d$.

Suppose $O\notin U$. Choose a type $1/2$ direction, whose $b$
singleton points fill its maximal centroid line $L$, and choose
another maximal line $L'$. All $b$ points of $U\cap L'$ lie outside
$L$ and are nonsingletons for the chosen direction. Since the direction
of $L'$ is of type $2$, these points belong to distinct trisecants.
Consequently $q\ge b+3b=q+3$, a contradiction. This proves (iii).

Let $\mathcal M$ be the set of maximal lines. As they all pass through
$O\in U$, their union contains $1+|\mathcal M|K$ points of $U$, so
$|\mathcal M|\le4$.
The reflection in Lemma~\ref{lem:type-two-reflections} associated
with any maximal line permutes the maximal lines. Exactly two of
them are fixed: the chosen line and its centroid line, which are
the two eigenspaces of the reflection. Thus $|\mathcal M|$ is even,
proving (iv).
\end{proof}

\begin{proposition}\label{prop:five}
In the range \eqref{eq:range}, if $p=5$ and $h>0$, then
$\delta=0$ and $U$ is affinely equivalent to $\UU_q$.
\end{proposition}
\begin{proof}
By Lemmas~\ref{lem:maxsize} and \ref{lem:two-lines}, it suffices
to exclude the case where the largest secant size is $b$, treated in
Lemma~\ref{lem:residual-geometry}.
For $P\in U\setminus\{O\}$, let $\tau_2(P)$ and $\tau_{1/2}(P)$ be
the numbers of type $2$ and type $1/2$ directions, respectively,
in which the line through $P$ is tangent to $U$. Every such direction
has centroid line $OP$. Lemma~\ref{lem:type-two-reflections} gives
$\tau_2(P)\le1$. Counting the other points of $U$ on the lines
through $P$ with directions in $D_U$, modulo five, gives
\[
q-1\equiv (|D_U|-h-\tau_2(P))+2(h-\tau_{1/2}(P))\pmod5.
\]
Since $|D_U|\equiv4+\delta\pmod5$, this is
\begin{equation}\label{eq:five-singleton-congruence}
\tau_2(P)+2\tau_{1/2}(P)\equiv h+\delta\pmod5.
\end{equation}
Equation~\eqref{eq:hcongruence} also gives
\begin{equation}\label{eq:five-h}
h\equiv
\begin{cases}
2\pmod5,&\delta=0,\\
0\pmod5,&\delta=1.
\end{cases}
\end{equation}

Suppose first that there are two maximal lines, $L_1,L_2$.
Their union contains $2K+1<q$ points of $U$. If
$P\in U\setminus(L_1\cup L_2)$, then $\tau_{1/2}(P)=0$, since every
type $1/2$ centroid line is maximal. When $\delta=0$,
\eqref{eq:five-singleton-congruence} gives
$\tau_2(P)\equiv2\pmod5$, contrary to $\tau_2(P)\le1$.
This excludes the case of two maximal lines when $\delta=0$.
When $\delta=1$, the same congruence gives $\tau_2(P)=1$. Thus each point
outside the two maximal lines has exactly one singleton direction.

Choose a type $1/2$ direction $d$, and relabel so that $L_d=L_1$.
Take $A\in(U\cap L_2)\setminus\{O\}$. The line $\ell$ through $A$
with direction $d$ is a trisecant, with its other two points $P,Q$
outside $L_1\cup L_2$. Indeed, all points of $U\cap L_1$ are
singletons in direction $d$, and $\ell\ne L_2$. Also $\ell$ avoids $O$.

Since $P,Q$ each have a unique singleton direction and $\ell$ is a
trisecant, Lemma~\ref{lem:index-comparison} forces $\ell$ to pass
through $O$, a contradiction.

It remains to consider four maximal lines. Their union contains
$1+4K=q$ points of $U$, so it contains all of $U$. For each maximal
line $L$, let $h_L$ count the type $1/2$ directions whose centroid
line is $L$. The centroid correspondence permutes the maximal lines.
For the maximal-line direction with centroid line $L$, its $K$
singleton points are precisely $(U\cap L)\setminus\{O\}$:
$O$ lies on its $b$-secant and is not a singleton point in this direction. Consequently,
for every $P\in(U\cap L)\setminus\{O\}$, one has
$\tau_2(P)=1$ and $\tau_{1/2}(P)=h_L$. Equations
\eqref{eq:five-singleton-congruence} and \eqref{eq:five-h} give
\begin{equation}\label{eq:five-hL}
h_L\equiv
\begin{cases}
3\pmod5,&\delta=0,\\
0\pmod5,&\delta=1.
\end{cases}
\end{equation}

We show that $h_L\le2$. Make $L$ the line $z=0$, and write the
other three maximal lines as $x=a_i z$, where $a_1,a_2,a_3$ are
distinct. The lines in any type $1/2$ direction with centroid line
$L$ have equations $x-tz=c$, with $c\in\F$, for some
$t\notin\{a_1,a_2,a_3\}$: the direction is not that of $L$, since
$L$ is its centroid line, and it differs from the directions of the
other three maximal lines, whose directions are of type $2$. A trisecant in this direction avoids $O$ and
the points of $U\cap L$, and has one point on each of the other
three maximal lines. For some $c\ne0$, those three points are
\[
\left(\frac{a_i c}{a_i-t},\frac{c}{a_i-t}\right),
\qquad i=1,2,3.
\]
By Lemma~\ref{lem:centroid}, their centroid lies on the centroid
line $L$, so its $z$-coordinate
$\frac{c}{3}\sum_{i=1}^3\frac1{a_i-t}$ vanishes. As $c\ne0$,
\[
\sum_{i=1}^3\frac1{a_i-t}=0.
\]
Writing $F(W)=\prod_{i=1}^3(W-a_i)$, this says $F'(t)/F(t)=0$, that
is, $F'(t)=0$. The polynomial $F'$ is quadratic, with leading
coefficient $3\ne0$ in characteristic five. Distinct type $1/2$
directions with centroid line $L$ have distinct values of $t$, so
$h_L\le2$.

Since $h>0$, some maximal line satisfies $1\le h_L\le2$.
This contradicts \eqref{eq:five-hL} in both cases. Hence the largest
secant size is not $b$, completing the proof.
\end{proof}

\section{Characteristic three}\label{sec:three-reflections}
Here assume $p=3$. Then $\delta=0$, so $|D_U|=m+2$,
and every secant meets $U$ in $2,5,8,\ldots$ points.
If $s$ is the number of tangents in a determined direction $d$,
subtracting twice the number $M$ of lines in that direction meeting
$U$ from the sum $q$ of their intersection numbers gives
\begin{equation}\label{eq:singleton3}
s=1+\sum_{\substack{\ell\text{ a secant}\\\mathrm{dir}(\ell)=d}}(|U\cap\ell|-2).
\end{equation}
In particular $s\equiv1\pmod3$ and $s\ge1$.

Translate $O$ to the origin from now on. To prove that $U$ lies on two
nonparallel lines, suppose the contrary. By Lemma~\ref{lem:large-secant},
the largest secant size then satisfies
\begin{equation}\label{three:eq:kbound}
k\le\frac{q+3}{4}.
\end{equation}
We work under this assumption until the contradiction at the end of
Section~\ref{sec:three-reflections}.

\subsection*{Centroid reflections}

\begin{lemma}\label{three:lem:reflection}
Under \eqref{three:eq:kbound}, for every $d\in D_U$, the reflection $r_d$
preserves $U$.
\end{lemma}

\begin{proof}
For this direction let $s$ count tangents, $b_2$ bisecants, and $t$
lines meeting $U$ in at least five points. We have
$M=s+b_2+t$, and \eqref{eq:singleton3} gives $s\ge1+3t$.
All singleton points lie on $L_d$, so $s\le k$. Therefore
\begin{align*}
b_2-t
&=M-s-2t\\
&\ge M-\frac{5s-2}{3}\\
&\ge \frac{q+1}{2}-\frac{5(q+3)/4-2}{3}
 =\frac{q-1}{12}>0.
\end{align*}
The reflection fixes every singleton point and swaps the two points
of $U$ on every bisecant, because their midpoint is on $L_d$. Consequently
\begin{equation}\label{three:eq:overlap}
|U\cap r_d(U)|\ge s+2b_2=M+(b_2-t)>M.
\end{equation}

Since $r_d$ is an invertible linear transformation, $U$ and $r_d(U)$
are $q$-sets each determining $(q+3)/2$ directions. Thus
\eqref{three:eq:overlap} and Corollary~\ref{cor:coincidence} give $r_d(U)=U$.
\end{proof}

\subsection*{Completion}

Every secant has size $2\pmod3$, and $|D_U|=(q+3)/2\equiv0\pmod3$.
Lemmas~\ref{three:lem:reflection} and~\ref{lem:reflection-secant}
therefore give a secant of size at least
\[
q-|D_U|+2=\frac{q+1}{2}=M,
\]
contrary to \eqref{three:eq:kbound}. Consequently $U$ lies on two
nonparallel lines, and Lemma~\ref{lem:two-lines} identifies it with
$\UU_q$. This proves Theorem~\ref{thm:main} in characteristic three
for $n\ge2$.

\section{Completion and the boundary case \texorpdfstring{$|D_U|=(q+p)/2$}{|D\string_U|=(q+p)/2}}
\begin{proof}[Proof of Theorem~\ref{thm:main}]
For $n=1$, the assertion is the Lov\'asz--Schrijver theorem
\cite{LS}. Suppose $n\ge2$. The characteristic-three case was
completed above. For $p>3$,
the assumption $|D_U|=(q+3)/2$ and \eqref{eq:hcongruence} give $h\equiv2\pmod p$,
so $h>0$. Proposition~\ref{prop:largep} or
Proposition~\ref{prop:five}, as appropriate, gives the classification.
\end{proof}

We next exclude the boundary case $|D_U|=(q+p)/2$ for $p>3$.

\Needspace{6\baselineskip}
\begin{proposition}\label{prop:equality}
Let $q=p^n$, where $p>3$ is prime and $n\ge2$. There is no $q$-set
in $\AG(2,q)$ determining exactly $(q+p)/2$ directions.
\end{proposition}
\begin{proof}
Suppose that such a set $U$ exists, and put $N=|D_U|=(q+p)/2$.
If $h>0$, Proposition~\ref{prop:largep} or
Proposition~\ref{prop:five} would make $U$ affinely equivalent to
$\UU_q$, contrary to $N>(q+3)/2$. Hence $h=0$, so every determined
direction has type $2$. In particular, every secant has size
$2\pmod p$. If $k$ is the largest secant size, then
Lemma~\ref{lem:type-two-singletons} gives
\begin{equation}\label{eq:boundary-kbound}
k\le K+1.
\end{equation}
Every centroid reflection preserves $U$, by
Lemma~\ref{lem:type-two-reflections}. Since $N\equiv0\pmod p$,
Lemma~\ref{lem:reflection-secant} gives
\[
k\ge q-N+2=\frac{q-p+4}{2}.
\]
This contradicts \eqref{eq:boundary-kbound}, because
\[
\frac{q-p+4}{2}-(K+1)
=\frac{(p-3)(q-p)}{2(p-1)}>0.
\]
\end{proof}

\begin{proof}[Proof of Theorem~\ref{thm:stability}]
If $|D_U|<(q+3)/2$, the established direction theorem
\cite{BBBS,Ball} gives the affine $\Fp$-linear alternative.
When $p=3$, the only remaining case is $|D_U|=(q+3)/2$, which is
classified by Theorem~\ref{thm:main}.
Suppose $p>3$ and $(q+3)/2\le |D_U|\le(q+p)/2$.
Proposition~\ref{prop:equality} excludes $|D_U|=(q+p)/2$, so the
common framework applies with
$0\le\delta<(p-3)/2$. If $h=0$, equation
\eqref{eq:hcongruence} implies
\[
p\mid 6+4\delta,\qquad 6\le6+4\delta<2p.
\]
There is no even positive multiple of the odd prime $p$ in this
interval. Thus $h>0$, and Propositions~\ref{prop:largep}
and \ref{prop:five} force $\delta=0$ and $U\simeq\UU_q$.
This proves the theorem and excludes the entire stated interval,
including $|D_U|=(q+p)/2$.
\end{proof}

Theorem~\ref{thm:stability} also strengthens Corollary~\ref{cor:coincidence}
and Lemma~\ref{lem:boundary-coincidence}.

\begin{corollary}
\label{cor:full-coincidence}
Let $q=p^n$, where $p>2$ is prime and $n\ge2$. Suppose that $U,V$ are
$q$-sets with $|D_U|,|D_V|\le(q+p)/2$. If
\[
|U\cap V|>\frac{q+1}{2},
\]
then $U=V$. The threshold is sharp.
\end{corollary}
\begin{proof}
Put $M=(q+1)/2$ and $K=(q-1)/(p-1)$. By
Theorem~\ref{thm:stability}, each set is either affine $\Fp$-linear
or affinely equivalent to $\UU_q$. If both are of the latter type,
Corollary~\ref{cor:coincidence} applies.
Otherwise, interchange the sets if necessary so that $U$ is affine
$\Fp$-linear. Every determined-direction line through a point of
$U$ contains at least $p$ points of $U$, whence $|D_U|\le K$.
Theorem~\ref{thm:stability} gives $|D_V|\le M+1$. Since $U\cap V$ contains
two distinct points, the direction sets have a common element, so
\[
|D_U\cup D_V|\le K+M\le q.
\]
Choose a direction determined by neither set as $\infty$. The reduced graph
polynomial of $U$ is a constant plus a linearized polynomial, of
degree at most $q/p\le M$. The reduced graph polynomial of $V$ has degree
at most $M$ as well: use the same observation in the linear case,
and Lemma~\ref{lem:gap} otherwise. Their difference has more than
$M$ roots, so it is zero. The example after
Corollary~\ref{cor:coincidence} proves sharpness.
\end{proof}

\Needspace{8\baselineskip}
\section*{Declaration on the use of AI}
The author used GPT-6 Astra, accessed through ChatGPT, in interactive
discussions to develop the mathematical arguments and exposition of this
manuscript. The author also used Claude Opus 5.5 (Anthropic), accessed
through the Claude app, to check the proofs and to suggest corrections
and improvements to the wording. The author takes full responsibility for the manuscript.

\Needspace{8\baselineskip}
\section*{Acknowledgements}
This paper was supported by the J\'anos Bolyai Research Scholarship of
the Hungarian Academy of Sciences. Project nos.\ 153080, 151504 and
152582 have been implemented with the support provided by the Ministry
of Culture and Innovation of Hungary from the National Research,
Development and Innovation Fund, financed under the ADVANCED 25,
EXCELLENCE 24 and SNN 25 funding schemes, respectively.

\end{document}